\documentclass[12pt,reqno]{amsart}

\usepackage{ccfonts}
\usepackage[T1]{fontenc}
\usepackage[cp1251]{inputenc}
\usepackage{amsmath,amsxtra,amssymb,eufrak}%{amscd}
\usepackage[all]{xy}
\usepackage{mathrsfs}
\usepackage{paralist}
\usepackage[active]{srcltx} %SRC Specials for DVI Searching

\newtheorem{thm}{Theorem}[section]
\newtheorem{lm}[thm]{Lemma}
\newtheorem{co}[thm]{Corollary}
\newtheorem{pr}[thm]{Proposition}

\theoremstyle{definition}

\newtheorem{df}[thm]{Definition}

\newtheorem{exm}[thm]{Example}
\newtheorem{rem}[thm]{Remark}
\newtheorem{rems}[thm]{Remarks}
\newtheorem{que}[thm]{Question}
\newtheorem{conj}[thm]{Conjecture}

\numberwithin{equation}{section}

\DeclareMathOperator{\LinC}{Lin_\CC}
\DeclareMathOperator{\wtLinC}{\wt{Lin}_{\CC}}

\newcommand*{\wh}{\widehat}
\newcommand*{\wt}{\widetilde}

\newcommand{\ptn}{\mathbin{\widehat{\otimes}}}

\newcommand{\GL}{\mathop{\mathrm{GL}}\nolimits}

\newenvironment{mycompactenum}{\pltopsep=5pt\begin{compactenum}[\upshape (i)]}%
{\end{compactenum}}
\newcommand*{\cA}{\mathscr A}

\newcommand{\cO}{\mathcal{O}}

\newcommand*{\cR}{\mathcal R}

\newcommand{\CC}{\mathbb{C}}

\newcommand{\R}{\mathbb{R}}
\newcommand{\Z}{\mathbb{Z}}
\newcommand{\N}{\mathbb{N}}

\newcommand*{\fn}{\mathfrak{n}}

\renewcommand{\le}{\leqslant}
\renewcommand{\ge}{\geqslant}

\let \al         =\alpha
\let \be         =\beta
\let \ga         =\gamma
\let \de         =\delta
\let \ep         =\varepsilon      

\let \io         =\iota
\let \ka         =\varkappa
\let \la         =\lambda
\let \si         =\sigma
\let \up         =\upsilon
\let \om         =\omega

\let \Ga         =\Gamma
\let \De         =\Delta
        
\let \La         =\Lambda

\let \phi         =\varphi

\title{On holomorphic reflexivity conditions for complex Lie groups
}
\author{O. Yu. Aristov}
\email{aristovoyu@inbox.ru}
\thanks{This work was supported by the RFBR grant no. 19-01-00447.}

\begin{document}
\begin{abstract}
We consider Akbarov's holomorphic version of the non-commutative Pontryagin duality for a complex Lie group. We prove, under the assumption that~$G$ is a Stein group  with finitely many   components, that (1)~the topological Hopf algebra of holomorphic functions on~$G$  is holomorphically reflexive if and only if~$G$ is linear; (2)~the dual cocommutative topological Hopf algebra of exponential analytic functional on~$G$ is holomorphically reflexive.
We give a counterexample, which shows that the first criterion cannot be extended to the case of infinitely many components. Nevertheless, we conjecture  that, in general, the question can be solved  in terms of the Banach-algebra linearity of~$G$.
\end{abstract}

 \maketitle

\section{Introduction}
\label{sec:intr}

We consider a holomorphic version of the non-commutative Pontryagin duality and examine conditions for the holomorphic reflexivity for complex Lie groups. The research of holomorphic duality and its extension from abelian
to general complex Lie groups (and far --- to quantum groups) was
initiated by Akbarov \cite{Ak08}. The current article continues the
line and gives some positive and negative results. We exhibit a counterexample showing that a claim and a more general conjecture made by Akbarov are too optimistic in general. Notwithstanding, the conjecture can be proved for Stein groups with finitely many components. Moreover, we formulate a possible necessary and sufficient condition for the general case.

\subsection*{Holomorphic Pontryagin duality}
In \cite{Ak08}, Akbarov made an observation that a holomorphic
version of the classical Pontryagin duality makes sense for
abelian complex Stein Lie groups. For such a group~$G$, one can consider the
dual group~$G^\bullet$ consisting, by definition, of holomorphic characters, i.e., homomorphisms to~$\CC^\times$,
the  multiplicative group of complex numbers. Then there are isomorphisms, which are analogous to the duality relations for~$\R$ and~$\mathbb{T}$ in the continuous version:
$$
 \CC^\bullet\cong \CC,\qquad(\CC^\times)^\bullet\cong \Z, \qquad\Z^\bullet\cong
 \CC^\times\,.
$$
Moreover, it is not hard to prove
a reflexivity result: \emph{If $G$ is a compactly generated
abelian Stein group, then the naturally defined holomorphic
homomorphism $G^{\bullet\bullet}\to G$ is a complex Lie group
isomorphism} [ibid., Theorem~7.2]. Two
restrictions accepted here  are essential by the following reasons.
First, we need sufficiently many holomorphic functions
(particularly, characters) on~$G$, so we consider only Stein
groups and do not consider such groups as $\CC/(\Z+i\Z)$. Secondly, we assume  that~$G$ is
compactly generated to endow~$G^\bullet$ with a Lie group structure  and
reject such groups as the countable  sum of copies of~$\Z$.

Beginning from the  observation above Akbarov made a step into the ``quantum-group
world'' with use of topological Hopf algebras. Before formulating his idea, which is basic to our
consideration, we briefly discuss an approach proposed by
Bonneau, Flato, Gerstenhaber and Pinczon in \cite{BFGP}.

\subsection*{The BFGP duality schema}
Recall that  a \emph{$\ptn$-algebra} (\emph{$\ptn$-coalgebra},
\emph{$\ptn$-bialgebra}, \emph{Hopf $\ptn$-algebra}) is an algebra (coalgebra,
bialgebra, Hopf algebra) in the symmetric monoidal category of
complete locally convex spaces endowed with the bifunctor
$(-)\ptn(-)$ of the complete projective tensor
product~\cite{Lit2,Pir_stbflat}. (Since $(-)\ptn(-)$ linearizes
jointly continuous bilinear maps, a $\ptn$-algebra is just a
complete locally convex algebra with jointly continuous
multiplication.) In this article, we consider only unital algebras.

In particular, a Hopf $\ptn$-algebra is defined as a $\ptn$-bialgebra $(H,m,u,\De,\ep)$ endowed with a continuous linear map~$S$ that satisfies the main antipode axiom: the diagram
\begin{equation*}
  \xymatrix{
 H\ptn H\ar[d]^{S\ptn 1}&\ar[l]_{\De}H\ar[d]^{u\ep}\ar[r]^{\De}
 &H\ptn H\ar[d]^{1\ptn S}\\
H\ptn H\ar[r]_m&H&H\ptn H\ar[l]^m\,, }
\end{equation*}
where $H$ is a $\ptn$-algebra  with respect to the multiplication $m$ and the unit $u\!:\CC\to H$, and a
$\ptn$-coalgebra with respect to the comultiplication $\De$ and the counit $\ep\!:H\to \CC$, commutes.

Following \cite{BFGP} (cf.~\cite{Lit2}) we say that a
$\ptn$-algebra ($\ptn$-coalgebra, $\ptn$-bialgebra, Hopf
$\ptn$-algebra)~$A$ is \emph{well-behaved} if it is either a
nuclear Fr\'echet  or (complete) nuclear (DF)-space. Whereas the definition
is somewhat artificial,  this class is very useful because of good
duality properties. Indeed, the strong dual space~$E'$ of a nuclear
Fr\'echet space~$E$  is a complete nuclear (DF)-space, and vice
versa. In both cases, for a well-behaved~$A$, there is a canonical topological isomorphism
$A'\ptn A'\cong (A\ptn A)'$, which allows us to endow $A'$ with a
structure of a well-behaved $\ptn$-coalgebra ($\ptn$-algebra, $\ptn$-bialgebra,
Hopf $\ptn$-algebra, respectively). Moreover,
$A\mapsto A'$ is a contravariant endofunctor between   the
corresponding categories.

Recall that for a complex Lie group $G$, the commutative $\ptn$-algebra $\cO(G)$
of holomorphic functions on $G$ is endowed with a Hopf
$\ptn$-algebra structure by setting
\begin{equation}\label{coomHopf}
 \De (f)(g, h) = f(gh),\quad \ep(f) = f(1),\quad  (Sf)(g) = f(g^{-1})\,,
\end{equation}
where $f\in\cO(G)$ and $g,h\in G$. (We use the identification
$\cO(G)\ptn \cO(G)\cong \cO(G\times G)$.) Since $\cO(G)$ is a
nuclear Fr\'echet space, the correspondence $G\mapsto\cO(G)$ can be
extended to a contravariant functor from the category of complex
Lie groups to the category of commutative well-behaved Hopf
$\ptn$-algebras. Also, the space
$\cA(G):= \cO(G)'$ of analytic functionals is also a well-behaved Hopf $\ptn$-algebra with respect to the convolution and the comultiplication, counit and antipode determined by
\begin{equation}\label{cocmHopf}
\De(\de_g)=\de_g\otimes\de_g,\quad\ep(\de_g)=1,\quad S(\de_g)=\de_{g^{-1}}\,,
\end{equation}
respectively, where $\de_g$ denotes the delta-function at $g\in G$.
Moreover, $G\mapsto\cA(G)$ is a covariant functor to the category of
cocommutative well-behaved Hopf $\ptn$-algebras.

This direct approach to duality  has advantages and sometimes works perfectly (see, e.g., \cite{Pir_stbflat}) but also has drawbacks. For example, the strong dual of $\cO(\CC)$ is
$\cA(\CC)$, which is isomorphic to $\cO_{exp}(\CC^\bullet)$, the algebra of holomorphic
functions of exponential type~$\CC^\bullet$, and obviously strictly smaller than $\cO(\CC^\bullet)$. So the strong
duality functor is not compatible with the holomorphic version of
the Pontryagin duality.

\subsection*{Akbarov's duality schema and its modification}

Recall that the \emph{Arens-Michael envelope} of  a $\ptn$-algebra~$A$ is its completion with respect to the topology determined by all possible continuous submultiplicative prenorms~$\|\cdot\|$, i.e.,  satisfying $\|ab\|\le\|a\|\,\|b\|$ for all $a,b\in A$ \cite{He93}. Since we assume that~$A$ has the identity~$1$,   we can assume that $\|1\|=1$ by Gelfand's Lemma; see
~\cite[Proposition~2.1.9]{Da00} or~\cite[Proposition~1.1.9]{Pa1}.
The Arens-Michael  envelope of a $\ptn$-algebra~$A$ is denoted by~$\wh A$. (The notation in  \cite{Ak08}  is~$A^\heartsuit$.) Pirkovskii observed that the Arens-Michael envelope is an endofunctor on the category
of Hopf $\ptn$-algebras  \cite[Proposition~6.7]{Pir_stbflat}, which obviously takes values in the subcategory of Arens-Michael Hopf algebras.

\begin{df}\label{Homduldef}
The \emph{holomorphic dual} of  an Arens-Michael,
well-behaved Hopf $\ptn$-algebra~$H$ is the Arens-Michael
Hopf $\ptn$-algebra
\begin{equation}\label{defhold}
H^\bullet\!:=(H')\sphat\,.
\end{equation}
\end{df}
Of course, $H^\bullet$ can be considered for any
well-behaved Hopf $\ptn$-algebra $H$ but we are interested in the
Arens-Michael case.

Our approach is close to that in \cite{Ak08} but it is different in two points.  First, Akbarov introduces his duality schema in the category of stereotype Hopf algebras but we work in more narrow context and use strong dual locally convex spaces instead of stereotype dual spaces --- that is sufficient  for the cases considered in the current article. Secondly,  we are interested in the duality between $\cO(G)$ and $\cA_{exp}(G)$  when~$G$ is a compactly generated complex Lie group. (The space $\cA_{exp}(G)$ is defined as the strong dual space of $\cO_{exp}(G)$; its elements are called  \emph{exponential analytic functionals}.) On the other hand, Akbarov considers the duality between $\cA(G)$ and $\cO_{exp}(G)$  applying the operations in the reverse order: first the Arens-Michael envelope and next the dual functor in the category of stereotype spaces \cite[Section~6.5, p.~549]{Ak08}.
Whereas our definition of duality (as well as the definition of reflexivity; see below) is formally
different it is the same in essence. But there are several reasons to consider the Arens-Michael Hopf $\ptn$-algebras, e.g., $\cO(G)$ and $\cA_{exp}(G)$, better than their strong duals.

(1)~It is easier to check that a $\ptn$-algebra is an Arens-Michael algebra than to check that a  $\ptn$-coalgebra is the strong dual of an Arens-Michael algebra.

(2)~$\cO(G)$ can be characterized in the spirit of the first Gelfand-Naimark theorem (see my paper \cite{AHHFG}).

(3) There is an analogy with the construction of the Pontryagin duality extended to non-commutative locally compact groups (as well as $C^*$-algebraic quantum groups). The first step in~\eqref{defhold}, the strong dual functor, is justified by the fact that there is a one-one correspondence between  holomorphic representations of a complex Lie group~$G$ on a Banach space~$X$ and non-degenerate continuous representations of $\cA(G)$ on~$X$ (see \cite{Lit} or \cite[(5.1)--(5.3)]{ArAnF}). The second step, the Arens-Michael enveloping functor, corresponds to  completion of $\cA(G)$ with respect to all prenorms associated with continuous Banach-space representations. Similarly, it is well known that
there is a one-one correspondence between strongly continuous unitary representations of a locally compact group~$G$ on a
Hilbert space~$H$ and non-degenerate (automatically continuous) $*$-representations of $L^1(G)$ on~$H$; see, e.g., \cite[Theorem II.10.2.2]{Bl06}. In general, $L^1(G)$, which is a dual Banach space to $C_0(G)$, is not a $C^*$-algebra but its $C^*$-envelope, which can be obtain by completion  with respect to all the prenorms associated with Hilbert-space $*$-representations, is traditionally considered as a dual object to~$C_0(G)$.

(4)~When we consider an arbitrary  well-behaved Hopf $\ptn$-algebra~$H$, the Hopf $\ptn$-algebra $H^\bullet$ is always well defined but there is no obvious reason for $\wh H$ to be well-behaved and so for $(\wh H)'$ to be well defined.

\subsection*{The definition of holomorphic reflexivity}
There is no
motive to claim that $H^\bullet$ is always well-behaved. But we want
to apply our duality functor twice, and so  an additional
restriction is necessary. In the  definition below,
we consider holomorphic duality and reflexivity in the more narrow
context of nuclear Fr\'echet Arens-Michael Hopf
$\ptn$-algebras  (in short, \emph{NFAM Hopf
algebras})  that allows us to use the
strong dual spaces as in the BFGP duality schema.

Note that $H\mapsto H^\bullet$ is a contravariant functor from the
category of Arens-Michael, well-behaved Hopf $\ptn$-algebras to
the category of Arens-Michael Hopf $\ptn$-algebras.
When restricted to the category of NFAM Hopf
algebras it can be considered as an extension  of the
holomorphic duality functor for abelian compactly generated Stein groups. Indeed, if $G$ is such a group, then
$\cO(G)^\bullet\cong \cO(G^\bullet)$
\cite[Theorem 7.2]{Ak08} and this isomorphism is functorial.

Suppose that $H$ is an  NFAM Hopf algebra. In general, we cannot guarantee that $H^\bullet$ is a Fr\'echet  space.
But when this is the case,  $H^\bullet$ is also an NFAM Hopf algebra because   $H^\bullet$  is
nuclear by Corollary~\ref{conuclAMe} in the Appendix.
Under the assumption that $H^\bullet$ is a Fr\'echet  space we have.
the homomorphism of
well-behaved Hopf $\ptn$-algebras $(H^\bullet)'\to H''$ given by the
application of the strong dual functor to the Arens-Michael
envelope homomorphism $H'\to H^\bullet$. Since  $H''\cong H $, we can write the
universal diagram for the Arens-Michael envelope of $(H^\bullet)'$:
\begin{equation}\label{defioH}
  \xymatrix{
(H^\bullet)' \ar[r]\ar[rd]&H^{\bullet\bullet}\ar@{-->}[d]^{\ka_H}\\
 &H\\
 }
\end{equation}
By \cite[Proposition~6.7]{Pir_stbflat}, $\ka_H$ is a Hopf $\ptn$-algebra
homomorphism.

\begin{df}\label{Homrelfdef}
Suppose that $H$ is an NFAM Hopf $\ptn$-algebra $H$ such that $H^\bullet$  is a
Fr\'echet space.
If  the Hopf $\ptn$-algebra homomorphism
$$
\ka_H\!:H^{\bullet\bullet}\to H
$$
defined in \eqref{defioH} is a topological isomorphism, then~$H$
is said to be
\emph{holomorphically reflexive}.
\end{df}

When we discuss below Akbarov's results and conjectures on holomorphic
duality and reflexivity, we always mean the reformulation in the sense of Definitions~\ref{Homduldef} and~\ref{Homrelfdef}.

The following question is a general challenge for research.
\begin{que}
Under which conditions an NFAM Hopf $\ptn$-algebra is
holomorphically reflexive?
\end{que}

The problem is rather wide and needs some additional
limitation. In this article, we restrict ourselves to commutative and
cocommutative NFAM Hopf $\ptn$-algebras arising from complex Lie groups. The investigation of
``quantum-group nature'' examples  is also initiated in
\cite[Section~8]{Ak08}. Other interesting examples of this type will be
considered in my subsequent paper.

\subsection*{Main results}
The basic aim of this article is to find conditions for holomorphic reflexivity of $\cO(G)$ and $\cA_{exp}(G)$ in the case when~$G$ is a compactly generated complex Lie group. It was proved by Akbarov in~\cite{Ak08} that $\cO(G)$ and $\cA_{exp}(G)$ are holomorphically reflexive  when $G$ is abelian (see the discussion above), discrete or isomorphic to $\GL_n(\CC)$ for some $n\in\N$. Also he claimed that $\cO(G)$ and $\cA_{exp}(G)$ are holomorphically
reflexive for a wide class of Lie groups, compactly generated Stein groups having algebraic
component of the identity \cite[Theorem~6.4]{Ak08}, and conjectured that this claim is also true when the component is linear \cite[pp.~540--541]{Ak17A}. Unluckily, the argument proposed in~\cite{Ak08} contains a gap (in the proof of \cite[Lemma~6.6]{Ak08}) and also a mistake. The gap was filled in my paper~\cite{ArAnF} and this allows us to prove the claim and conjecture for connected groups. Moreover, we show in the current article that, \emph{under the additional assumption that $G$~is a Stein group  with finitely many components, $\cO(G)$ is holomorphically reflexive if and only if~$G$ is linear} (Theorem~\ref{OGhrcon1}) \emph{and  $\cA_{exp}(G)$ is always holomorphically reflexive} (Corollary~\ref{AeGhr}).
Most of the preparatory work  was done for connected groups in~\cite{ArAnF} and, in this case,
it is not hard to deduce the criterion from~\cite[Theorem~5.12]{ArAnF}. But the argument for groups with finitely many components includes some additional results.

On the other hand, the assumption of \cite[Theorem~6.4]{Ak08}, namely, $G_0$, the component of the identity, is algebraic, does not imply the holomorphic reflexivity in the general case of infinitely many components. Indeed, in the key step of Akbarov's argument
(proof of \cite[Theorem~6.3]{Ak08}) the following implication is used: if~$G$ is a compactly generated Stein group and  the range of $\cO_{exp}(G_0)\to \cO(G_0)$ is dense, so is  the range of $\cO_{exp}(G)\to \cO(G)$.
But this is not the case. To explain the reason we note that, technically, the question of
holomorphic reflexivity of $\cO(G)$ (and, to some extent of $\cA_{exp}(G)$)
is reduced to an auxiliary question: whether or not
$\cO_{exp}(G)\to \cO(G)$ and $\cA(G)\to \cA_{exp}(G)$ are
Arens-Michael envelopes? (See Proposition~\ref{AMeqHR} and the proof of Theorem~\ref{AeGhr-1} below.) The answer to the second part of the question is positive: if~$G$ is a Stein group, then  the naturally defined map
$\cA(G)\to \cA_{exp}(G)$ is an Arens-Michael envelope
\cite[Theorem~6.2]{Ak08}.

The first part of the question is more difficult. In this paper, we repair Akabrov's argument in the case when the group  has finitely many components and show that, under this assumption $\cO_{exp}(G)\to \cO(G)$, is an Arens-Michael envelope if and only if $G$~is linear (Theorem~\ref{OGhrcon}).  The necessity is almost straightforward; cf. \cite[Theorem 5.3(C)]{ArAnF}. The main complexity is in the sufficiency: although $\cO_{exp}(G)$ is dense in $\cO(G)$ for every linear~$G$ \cite[Theorem~5.9]{Ak08}, this fact does imply immediately that the embedding is an Arens-Michael envelope. Fortunately,  there is a workaround --- we can find a subalgebra in $\cO_{exp}(G)$ that has $\cO(G)$ as an Arens-Michael envelope. Specifically, a linear group admits a structure of affine algebraic variety and we can use the algebra $\cR(G)$ of all regular functions. The only known way to show that $\cR(G)\subset \cO_{exp}(G)$ is to describe explicitly the structure of the latter algebra.  This approach requires some technical results, which were proved  in \cite{ArAMN,ArAnF} for connected groups. We use them in the proof of Corollary~\ref{expclin-1}, which refers to
\cite[Theorem~5.12(B)]{ArAnF}.

In the case when $G$ is compactly generated but the number of   components is infinite, the answer for the question can be negative even if the   component of the identity is  algebraic, as one can see from the following example.

\begin{exm}\label{countex}
First consider the following subgroup
\begin{equation}\label{3Hei}
H\!:=\left\{\begin{pmatrix}
 1& m& b\\
 0&1 & n\\
 0 & 0&1
\end{pmatrix}:\,\,m,n\in \Z_,\,b\in\CC\right\}
\end{equation}
of  the $3$-dimensional complex Heisenberg group
and the discrete central subgroup of~$H$  given by
$$
N\!:=\left\{\begin{pmatrix}
1& 0& k\\
 0&1 & 0\\
 0 & 0&1
\end{pmatrix}:\,\,k\in\Z\right\}\,.
$$
The quotient $G\!:=H/N$ is identified  with $\CC^\times\times\Z^2$ endowing with the
group law
$$
 (z,m,n)\cdot (z',m',n')\!:= (zz'e^{nm'},m+m',   n+n')\,.
$$
The same argument as for  the quotient of the Heisenberg group  (see \cite[Introduction]{ArAnF})
shows that the coordinate
function $(z,m,n)\mapsto z$ is not of exponential type (see the definition in \S\,\ref{sec:AMenv}). Moreover, any holomorphic function of exponential type on~$G$ is constant on the   component~$G_0$ of the identity. (Note that~$G_0$  is isomorphic to $\CC^\times$.) So the range of $\cO_{exp}(G)\to \cO(G)$ is not dense and therefore this homomorphism is not an  Arens-Michael envelope. Moreover, $\cO(G)$ is not holomorphically reflexive; see details in Corollary~\ref{Ctimescoin} below. On the other hand,~$G_0$ is linear and, moreover, algebraic and the range of
$\cO_{exp}(G_0)\to \cO(G_0)$ is obviously dense.
\end{exm}

\subsection*{Conjectures}
The counterexample above demonstrates that the holomorphic reflexivity of $\cO(G)$ does not follow from the linearity of the group or its component of the identity. To find a general sufficient condition we suggest to examine the Banach-algebra linearity instead. Recall that the linearizer $\LinC(G)$ of a complex Lie group~$G$ (which the intersection of the kernels of all finite-dimensional holomorphic representations) is a ``measure of non-linearity'' of~$G$. A careful analysis of the proofs of our main results and Example~\ref{countex} shows that, for non-connected groups, it is natural to consider the generalized linearizer of~$G$: the intersection
of the kernels of all holomorphic homomorphisms
of~$G$ to the groups of invertible elements of Banach algebras. (The notation is $\wtLinC(G)$).
Besides being of independent interest, this subgroup enables us with a natural terminology to formulate hypotheses on the holomorphic reflexivity for general compactly generated groups. It is not hard to show that, for  a Stein group~$G$, the subgroup $\wtLinC(G)$  is trivial provided that $\cO(G)$ is
holomorphically reflexive  (see Proposition~\ref{tLinness} below).

\begin{conj}\label{mainconj}
If $G$ is a compactly generated Stein group, then the condition $\wtLinC(G)=\{1\}$ is not only necessary but also sufficient for the holomorphic reflexivity of~$\cO(G)$.
\end{conj}

At the moment, we know that the conjecture is true under any of the following additional assumptions.

(1)~$G$ is abelian \cite{Ak08}.

(2)~$G$ is  zero-dimensional, i.e., discrete, \cite{Ak08}.

(3)~$G$  has finitely many components (Theorem~\ref{OGhrcon1}).

Moreover, there are examples when $G$ is not abelian and not discrete with infinitely many components
but nevertheless $\cO(G)$ is holomorphically reflexive (Example~\ref{infcomphr}).
It follows from Theorem~\ref{AeGhr-1} that if Conjecture~\ref{mainconj} is true in general then so is the following conjecture.

\begin{conj}\label{mainconj2}
If $G$ is a compactly generated complex Lie group, then $\cA_{exp}(G)$ is
holomorphically reflexive.
\end{conj}

\section{Banach-algebra linearizer}
\label{sec:genlin}

Recall that  the intersection of the
kernels of all finite-dimensional holomorphic representations
of a complex Lie group~$G$  is called the \emph{linearizer} of~$G$ and denoted by
$\LinC(G)$ \cite[Definition~15.2.13]{HiNe}.
We also introduce $\wtLinC(G)$ as the intersection of the kernels
of all holomorphic homomorphisms of~$G$ to  groups of invertible elements of
Banach algebras. It can be regarded as a useful generalization of the linearizer; see an alternative formulation of the reflexivity condition in Theorem~\ref{OGhrcon1}.
In this section we discuss  basic properties of~$\wtLinC(G)$.

The following theorem is  a reformulation of~\cite[Theorem~2.2]{ArAnF}.
\begin{thm}\label{holhomline}
If a complex Lie group $G$ is connected, then
$$\LinC(G)=\wtLinC(G)\,.$$
\end{thm}
It is obvious that $\wtLinC(G)\subset \LinC(G)$ for any complex Lie group. So we  immediately have a useful corollary. (Here and
everywhere we denote by~$G_0$ the component of the identity
in~$G$.)
\begin{co}\label{LwtL}
If $G$~is a complex Lie group, then
 $$\LinC(G_0)=\wtLinC(G_0)\subset \wtLinC(G)\subset \LinC(G)\,.$$
\end{co}

It is well known that the linearizer of a connected complex Lie group is contained in the centre. A similar result holds for $\wtLinC(G)$.
\begin{lm}\label{lemmcenter}
If $G$~is a complex Lie group, then $\wtLinC(G)$ is a closed normal subgroup that is contained in the centre of~$G_0$.
\end{lm}
\begin{proof}
It is evident that $\wtLinC(G)$ is  closed and normal.

The regular representation of the discrete group $G/G_0$ in the
Hilbert space $\ell^2(G/G_0)$ is faithful and automatically
holomorphic, so $\wtLinC(G)\subset G_0$.  On the other hand, the kernel of the adjoint representation coincides with the centralizer of~$G_0$ in~$G$ \cite[Lemma 9.2.21]{HiNe}. Thus $\wtLinC(G)$ is contained in the centre of $G_0$.
\end{proof}

\begin{rem}
In the general non-connected case it is possible that
$\wtLinC(G)\ne \LinC(G)$ or $\wtLinC(G_0)\ne\wtLinC(G)$. Indeed, using the regular representation as in Lemma~\ref{lemmcenter}, we have
$\wtLinC(\Ga)=\{1\}$ for every discrete group~$\Ga$. However there are finitely
generated discrete groups that are not complex linear, i.e.,
$\LinC(\Ga)\ne\{1\}$. E.g., each discrete group that is complex
linear should be residually finite by Mal'cev's Theorem (see a
short proof in~\cite{Ni13}).

For the group considered in Example~\ref{countex}, we have
$\wtLinC(G_0)\ne\wtLinC(G)$ because $G_0$~is linear but
$\wtLinC(G)\ne\{1\}$.
\end{rem}

We now prove that the conclusion of Theorem~\ref{holhomline} holds for
a wider class of Lie groups.
\begin{thm}\label{Linfcc}
If a  complex Lie group $G$ has finitely many
components, then
$$\LinC(G_0)=\wtLinC(G)=\LinC(G)\,.$$
\end{thm}
For a complex manifold $M$ we denote by $\cO(M)$ the locally convex space of holomorphic functions on~$M$ (with the topology of uniform convergence on compact subsets) and by $\cA(M)$ the strong dual space of $\cO(M)$.

\begin{proof}
By Corollary~\ref{LwtL}, it suffices to show that
$\LinC(G)\subset\LinC(G_0)$. We need to prove that for any finite-dimensional
holomorphic representation~$\pi_0$ of~$G_0$ there is a finite-dimensional
holomorphic representation~$\pi$ of~$G$ such that $\pi(h)=1$ implies $\pi_0(h)=1$ when $h\in G_0$.

Fix~$\pi_0$, denote by~$X$ the
representation space of~$\pi_0$, and consider the induced
representation~$\pi$ of~$G$, i.e., the representation assosiated
with the $\cA(G)$-module $Y\!:=\cA(G)\otimes_{\cA(G_0)} X$. Here
$\cA(G)$ is considered as a right $\cA(G_0)$-module with respect to an algebra homomorphism
$\cA(G_0)\to \cA(G)$ induced by the group homomorphism $G_0\to G$.

We can consider each coset $\ga\in G/G_0$ as a complex manifold.
It is not hard to see that
$$
{\cA}(G)\cong \bigoplus_{\ga\in G/G_0}{\cA}(\ga)\,
$$
as right $\cA(G_0)$-modules.
Since $X$ is finite dimensional and $G/G_0$ is finite, the $\cA(G)$-module $Y$ is a finite dimensional linear space. Endowing $Y$ with the topology of the projective tensor product  we make it a Banach $\cA(G)$-module. Hence
$\pi$~is a holomorphic representation of~$G$ on~$Y$.

Since the counit $\ep\!:\cA(G)\to \CC$ is a quotient homomorphism of algebras,  $\CC$ is a right $\cA(G)$-module with respect to the multiplication determined by $\la\cdot \de_g\!:=\la$ (here $\de_g$ is the delta-function at~$g\in G$). Therefore $\ep$ is a right $\cA(G)$-module morphism and so is a right $\cA(G_0)$-module morphism. Hence the linear map $\ep\otimes 1\!:Y\to X\!:\nu\otimes x\mapsto  \ep(\nu)x $ is well defined.

Suppose that $h\in G_0$ and $x\in X$. Then $\de_h\cdot(1\otimes x)=\de_h\otimes x=1\otimes \de_h\cdot x$.
If, in addition,  $\pi(h)=1$, then $\de_h\cdot(1\otimes x)=1\otimes x$. Applying $\ep\otimes 1$ to the equality $1\otimes \de_h\cdot x= 1\otimes x$, we have $\de_h\cdot x=x$.  Since $x$ is arbitrary, we obtain
$\pi_0(h)=1$. Thus, $\LinC(G)\subset\LinC(G_0)$.
\end{proof}

\section{The Arens-Michael envelope of $\cO_{exp}(G)$}
\label{sec:AMenv}

In this section, we consider Hopf $\ptn$-algebras $\cO_{exp}(G)$ and $\cA_{exp}(G)$ for a compactly generated complex Lie group~$G$, establish a relation between holomorphic reflexivity of $\cO_{exp}(G)$ and a property of the Arens-Michael envelope and prove this property in the case of finite number of components.

Recall that a holomorphic function $f$ on a complex Lie group $G$ is said to be of \emph{exponential type} if there is
a locally bounded submultiplicative weight $\om\!: G \to [1,+\infty)$ (i.e.,
$\om(gh)\le  \om(g) \om(h)$   for all $g,h\in G$)
such that $|f(g)|\le \om(g)$  for each
$g\in G $ \cite[Section~4(c)]{Ak08}. The vector space of  holomorphic functions of exponential type on~$G$ is
denoted by $\cO_{exp}(G)$. Being endowed with the corresponding inductive topology  it is a complete locally convex space.

We are interested in the case when $G$~is compactly generated. Under this assumption $f\in \cO_{exp}(G)$
if and only if there are $K,C>0$ such that
\begin{equation}\label{exptwlf}
|f(g)|\le K e^{C\ell(g)}\qquad \text{for each $g\in G$,}
\end{equation}
where $\ell$ is a \emph{word length function}, i.e., $\ell(g)\!: = \min \{ n \!: \, g \in U^{n} \}$ for some relatively compact generating set~$U$; see  \cite[Theorem~5.3]{Ak08} or  \cite[Proposition~2.8]{ArAMN}.

We denote the strong dual space $\cO_{exp}(G)$ by  $\cA_{exp}(G)$. In the initial definition of Akbarov, $\cA_{exp}(G)$ is considered with another topology, namely, the topology of uniform convergence on totally bounded subsets of $\cO_{exp}(G)$. But the latter space is nuclear \cite[Theorem~5.10]{Ak08} and so has the Heine-Borel property and hence these topologies coincide.\footnote{Note also that in \cite{Ak08} $\cA(G)$ and  $\cA_{exp}(G)$ are denoted by $\cO^\star(G)$ and $\cO_{exp}^\star(G)$, respectively.}

Recall that a Hopf $\ptn$-algebra  is  said to be \emph{well-behaved} if it is either a
nuclear Fr\'echet space or a nuclear (DF)-space (see the discussion in Introduction).

\begin{pr}\label{OAHa}\emph{(cf.~\cite[Theorem 5.12(1)]{Ak08})}
If $G$ is a compactly generated complex Lie group, then  $\cO_{exp}(G)$ is a well-behaved Hopf $\ptn$-algebra with respect to the pointwise multiplication and the  comultiplication, counit, and antipode defined in~\eqref{coomHopf}.
\end{pr}
\begin{proof}
First, note that it follows from  \cite[Theorem~5.10]{Ak08} that $\cA_{exp}(G)$
is a nuclear Fr\'echet space. Then  $\cO_{exp}(G)$  is a nuclear (DF)-space and, by
\cite[Proposition~5.5(B)]{ArAnF}, the natural map
$$\cO(G)\times\cO(G)\to \cO(G\times G)$$ induce a topological isomorphism $$\cO_{exp}(G)\ptn \cO_{exp}(G)\cong \cO_{exp}(G\times G)\,.$$

We claim that under this identification the operations in~\eqref{coomHopf} makes $\cO_{exp}(G)$ a Hopf $\ptn$-algebra.
The only non-trivial assertions are that  $\cO_{exp}(G)$ is a  $\ptn$-algebra and~$\De$ in~\eqref{coomHopf} is a well-defined map $\cO_{exp}(G)\to \cO_{exp}(G\times G)$. The former is proved in \cite[Lemma~5.2]{ArAnF}.
To show the latter take for $f\in \cO_{exp}(G)$ a locally bounded submultiplicative weight~$\om$ such that $|f(g)|\le \om(g)$ for all $g\in G$. Then $\om\times \om\!:(g,h)\mapsto \om(g)\om(h)$ is a locally bounded submultiplicative weight on $G\times G$ such that $$
|\De(f)(g,h)|=|f(gh)|\le \om(gh)\le (\om\times \om)(g,h)
$$
for all $g,h\in G$.  Thus $\De(f)\in \cO_{exp}(G\times G)$.
\end{proof}

Since the Hopf $\ptn$-algebra $\cO_{exp}(G)$  is well-behaved, its strong dual $\cA_{exp}(G)$  is also a well-behaved Hopf $\ptn$-algebra with dual operations (the multiplication corresponds to the comultiplication, the unit to the counit and vice versa, the antipode  to the antipode) \cite[Proposition~1.3]{BFGP}. Restricted to the group Hopf algebra $\CC G$, which is dense in  $\cA_{exp}(G)$, the operations  has the form given in~\eqref{cocmHopf} (we identify elements of~$G$ with delta-functions).

The following result from \cite{Ak08} brings out the importance of $\cA_{exp}(G)$ and $\cO_{exp}(G)$  in our considerations.
\begin{thm}\label{AMEAG} \cite[Theorem 6.2]{Ak08}
If $G$ is a compactly generated complex Lie group, then
the natural homomorphism $\cA(G)\to\cA_{exp}(G)$ is an Arens-Michael envelope.
\end{thm}

Summarizing the results above we have the following assertion.

\begin{pr}\label{OAHb}\emph{(cf.~\cite[Theorem 5.12(2)]{Ak08})}
If $G$ is a compactly generated  complex Lie group,
then  $\cA_{exp}(G)$  is a nuclear Fr\'echet (hence well-behaved)  Arens-Michael Hopf
algebra that  is dual to $\cO_{exp}(G)$ in the sense of Bonneau, Flato, Gerstenhaber and Pinczon. The multiplication is just the convolution whereas the comultiplication, counit and antipode are determined by~\eqref{cocmHopf}.
\end{pr}

Any holomorphic homomorphism
$\phi\!:G\to H$  of compactly generated complex Lie group induces the Hopf $\ptn$-algebra
homomorphism $\wt\phi\!:\cO_{exp}(H)\to \cO_{exp}(G)$
given by $[\wt\phi(f)](g)\!:=f(\phi(g))$.  Moreover, the strong dual map $\wt\phi'\!:\cA_{exp}(G)\to \cA_{exp}(H)$ is also a Hopf
$\ptn$-algebra homomorphism.

The following theorem motivates us to consider $\wtLinC(G)$ instead
of  $\LinC(G)$ in the case when $G$ is not connected.

\begin{thm}\label{redliexpf}
Let $G$ be a compactly generated complex Lie group and let
$\pi_0\!:G\to G/\wtLinC(G)$ be the quotient homomorphism. Then

{\em (A)}  $\wt\pi_0'\!:\cA_{exp}(G)\to \cA_{exp}(G/\wtLinC(G))  $
is a Hopf $\ptn$-algebra isomorphism.

{\em (B)}  $\wt\pi_0\!:\cO_{exp}(G/\wtLinC(G))\to \cO_{exp}(G) $ is
a Hopf $\ptn$-algebra isomorphism.

{\em (C)}   Each function in  $\cO_{exp}(G)$ is constant on
cosets of $\wtLinC(G)$.
\end{thm}
\begin{proof}
Since $\wt\pi_0'$ and $\wt\pi_0$ are Hopf $\ptn$-algebra homomorphisms, it is sufficient to prove that they are topologically isomorphisms. But, by Theorem~\ref{AMEAG}, $\cA_{exp}(G)\cong \wh\cA(G)$, so we can apply exactly the same argument as in the proof of \cite[Theorem~5.3]{ArAnF} for $\LinC(G)$ in the connected case.
\end{proof}

\begin{co}\label{coredliexpf}
Let $H$ be a subgroup of  a compactly generated complex Lie group~$G$. Then $H\subset \wtLinC(G)$ if and only if  each function in  $\cO_{exp}(G)$ is constant on~$H$.
\end{co}
\begin{proof}
The necessity follows from Part~(C) of Theorem~\ref{redliexpf}. To prove the sufficiency
note that each function in
$\cO_{exp}(G)$ is a coefficient of a holomorphic representation of~$G$
in some Banach space (see, e.g.,
\cite[Proposition~5.1]{ArAnF}). Since such functions are constant on~$H$, the corresponding representation maps each element of~$H$ to the identity operator.
\end{proof}

By using the following result we can reduce our question on the
holomorphic reflexivity to a question on the Arens-Michael
envelope.

\begin{pr} \label{AMeqHR}
Let  $G$ be a compactly generated complex Lie group.
Then $\cO(G)$ is holomorphically reflexive if and only if the embedding $\cO_{exp}(G)\to \cO(G)$ is an Arens-Michael
envelope.
\end{pr}
\begin{proof}
By the definition of the holomorphic duality functor, $\cO(G)^{\bullet}\cong \wh\cA(G)$, which in turn is topologically isomorphic to $\cA_{exp}(G)$ (Theorem~\ref{AMEAG}).

Note that $\cA_{exp}(G)$ is a
nuclear Fr\'echet space \cite[Theorems~5.10]{Ak08}. So it is reflexive as a locally convex space \cite[Theorem~3.7.12]{BS12}, therefore  $\cA_{exp}(G)'=\cO_{exp}(G)''\cong\cO_{exp}(G)$. Thus
$\cO(G)^{\bullet\bullet}\cong \wh\cO_{exp}(G)$ and then $\cO_{exp}(G)\to \cO(G)$ is an Arens-Michael
envelope if and only if the canonical map $\cO(G)^{\bullet\bullet}\to\cO(G)$ is a topological isomorphism.
\end{proof}

Now we discuss the Arens-Michael
envelope of $\cO_{exp}(G)$.

\begin{pr}\label{expdis}
If $\Ga$ is a  finitely generated discrete group, then  the
natural embedding $\cO_{exp}(\Ga)\to \cO(\Ga)$ is  an
Arens-Michael envelope.
\end{pr}
 \begin{proof}
In fact, the result is proved in \cite{Ak08}; we give a sketch. Since
$\cO(\Ga)$ is the algebra of all functions with the topology of
uniform convergence on finite subsets, it suffices to show that
for each continuous submultiplicative  prenorm $\|\cdot\|$ on
$\cO_{exp}(\Ga)$ there are a finite subset~$S$ in~$\Ga$ and $C>0$
such that $\|f\|\le C  \max_{\ga\in S} |f(\ga)|$.

For  $\ga\in\Ga$ denote by~$\chi_\ga$ the characteristic function
of $\{\ga\}$. Since $\Ga$~is discrete,   $\chi_\ga\in
\cO_{exp}(\Ga)$. By \cite[Lemma~6.3]{Ak08}, the set
$S\!:=\{\ga\in\Ga:\,\|\chi_\ga\|>0\}$ is finite. For any
$f\in\cO_{exp}(\Ga)$ the series $\sum_{\ga\in \Ga}
f(\ga)\,\chi_\ga$ converges to $f$ in   the topology of
$\cO_{exp}(\Ga)$ \cite[Lemma~6.1]{Ak08}. So we have
$$
\|f\|\le \sum_{\ga\in \Ga} |f(\ga)|\,\|\chi_\ga\|=\sum_{\ga\in S}
|f(\ga)|\,\|\chi_\ga\|\le C\,  \max_{\ga\in S} |f(\ga)|,
$$
where $C\!:=\sum_{\ga\in S} \,\|\chi_\ga\|$.
\end{proof}

\begin{pr}\label{iodens}
Let $G$ be a  complex Lie group  and~$G_1$
a normal complex Lie subgroup of~$G$ containing  the   component of
the identity. Suppose that~$G_1$ is linear. If  the  continuous homomorphism of $\ptn$-algebras $\psi_1\!:\cO_{exp}(G)\to \cO_{exp}(G_1)$ induced by the embedding $G_1\to G$
has dense range, then so is  $\io\!:\cO_{exp}(G)\to \cO(G)$.
\end{pr}
\begin{proof}
Put $\Ga\!:= G/G_1$ and
for any $\ga\in \Ga$ denote by $\chi_\ga$ the characteristic
function of  $\ga$ (now as a subset of $G$).
Denote by ${}_gf$ the left shift a function $f$ on $g\in G$. Fix
$g_\ga\in \ga$ for each $\ga\in S$. Since every function in
$\cO_{exp}(G)$ is a coefficient of a holomorphic representation of~$G$
in some Banach space (see, e.g.,
\cite[Proposition~5.1]{ArAnF}), $\cO_{exp}(G)$ is invariant under shifts. So for any  finite subset $S$  of $G/G_1$ we have a
continuous homomorphism
\begin{equation}\label{defpsi}
\psi\!:\cO_{exp}(G)\to
\cO_{exp}(G_1)^{|S|}\!:f\mapsto({}_{g_\ga}f\chi_{G_1})\,.
\end{equation}

Note that $\psi_1$  is the restriction map, i.e.,
$\psi_1(f)= f\chi_{G_1}$. Hence,
since the range of $\psi_1$
is dense, so is the range of $\psi$. Further, since $G_1$ is linear,
\cite[Theorem~5.9]{Ak08} implies that the range of $\cO_{exp}(G_1)\to
\cO(G_1)$ is dense; so is the range of $\cO_{exp}(G_1)^{|S|}
\to \cO(G_1)^{|S|}$.

Finally, note that any $f\in \cO(G)$ can be
approximated by $\sum_{\ga\in S}f \chi_\ga$ for some finite subset~$S$ of~$G/G_1$. Thus the range of~$\io$ is dense.
\end{proof}

\begin{pr}\label{Akle}
Let $G$ be a compactly generated complex Lie group and  $G_1$
a normal complex Lie subgroup containing  the   component of
the identity. Suppose that the range of $\io\!:\cO_{exp}(G)\to \cO(G)$
is dense. If the natural embedding $\cO_{exp}(G_1)\to \cO(G_1)$ is
an Arens-Michael envelope, then so is $\cO_{exp}(G)\to \cO(G)$.
\end{pr}
\begin{proof}
(The following argument is contained implicitly in \cite{Ak08}.) We use the notation from the proof of Proposition~\ref{iodens}.
It is sufficient to show that, for a Banach algebra $B$ with a
submultiplicative norm $\|\cdot\|$,  any  continuous homomorphism
$\phi\!:\cO_{exp}(G)\to B$ factors uniquely on
$\io\!:\cO_{exp}(G)\to \cO(G)$.

Denote by $\si$ the  continuous homomorphism $\cO_{exp}(\Ga)\to \cO_{exp}(G)$ induced
by the quotient map $G\to\Ga$, where $\Ga\!:= G/G_1$.
Note that $\Ga$ is
discrete and finitely generated; so Proposition~\ref{expdis}
implies that $\phi\si$ factors on  $\cO_{exp}(\Ga)\to \cO(\Ga)$.
Since $\cO(\Ga)$ is just the algebra of all functions on $\Ga$ endowed
with the topology of uniform convergence on finite subsets,  there
are a finite subset~$S$ of~$\Ga$ and $C>0$ such that
\begin{equation*}
\|\phi\si(h)\|\le C\,\max_{\ga\in S}|h(\ga)|\qquad \text{for each $h\in
\cO_{exp}(\Ga)$.}
\end{equation*}
In particular, since the characteristic function $\chi_\ga$ equals the image  of the delta-function at~$\ga$ under~$\si$, we have
$\|\phi(\chi_\ga)\|=0$ for $\ga\notin S$. By \cite[Lemma~6.2]{Ak08}, the series
$\sum_{\ga}\de_\ga$ converges to~$1$ in $\cO_{exp}(\Ga)$, hence $f=\sum_{\ga}f\chi_\ga$ for any
$f\in \cO_{exp}(G)$. Therefore,
\begin{equation}\label{phiffis}
\phi(f)=\sum_{\ga\in S }\phi(f\chi_\ga)\qquad\text{for each $f\in
\cO_{exp}(G)$.}
\end{equation}

Consider the homomorphism $\psi$ defined in~\eqref{defpsi} and note that
the same formula determines a continuous
homomorphism $\psi'\!:\cO(G)\to \cO(G_1)^{|S|}$.
Put
$$
\up\!: \cO_{exp}(G_1)^{|S|}\to B\!:(f_\ga)\mapsto \sum_{\ga\in S}
\phi({}_{g_\ga^{-1}}f_\ga \chi_\ga)\,.
$$
It is clear that \eqref{phiffis} implies $\phi=\up\psi$.

Note that the Arens-Michael envelope functor
commutes with finite products. By the assumption, $\cO_{exp}(G_1)\to \cO(G_1)$ is an
Arens-Michael envelope, so is $\cO_{exp}(G_1)^{|S|}\to\cO(G_1)^{|S|}$. Therefore
$\up$ factors on this map. Thus, we obtain the diagram
 \begin{equation*}
   \xymatrix{
 \cO_{exp}(G) \ar[r]^{\io}\ar[dd]_{\psi}\ar[drr]_{\phi} & \cO(G)\ar[dd]^{\psi'}\ar[dr]^{\al} &\\
 &&B\\
\cO_{exp}(G_1)^{|S|} \ar[r]\ar[urr]^{\up}& \cO(G_1)^{|S|}
\ar[ur]_{\be} &
 }
\end{equation*}
where $\be$ is determined by the universal property and
$\al\!:=\be\psi'$. Since the square and the big triangle are commutative, we have
$\phi=\al\io$.

By the hypothesis, the range of $\io$ is
dense in $\cO(G)$,
hence $\io$ is an epimorphism. Thus, the homomorphism $\al$ satisfying
$\phi=\al\io$ is unique.
\end{proof}

\begin{co}\label{expclin-1}
Let  $G$ be a compactly generated complex Lie group such that the component~$G_0$ of the identity is linear. If the homomorphism $\psi_0\!:\cO_{exp}(G)\to
\cO_{exp}(G_0)$ induced by the embedding $G_0\to G$ has dense
range, then the embedding $\io\!:\cO_{exp}(G)\to \cO(G)$ is an
Arens-Michael envelope.
\end{co}
\begin{proof}
Since $G_0$ is linear, it follows from Proposition~\ref{iodens}
that~$\io$ has dense range. Moreover,   \cite[Theorem~5.12(B)]{ArAnF} implies that the natural embedding $\cO_{exp}(G_0)\to
\cO(G_0)$ is an Arens-Michael envelope because $G_0$ is linear and
connected. Being compactly generated, $G$ satisfies to the
hypothesis of Proposition~\ref{Akle}, so~$\io$ is an Arens-Michael
envelope.
\end{proof}

We use the following notation. For positive functions $\tau_1$ and $\tau_2$ on a set $X$, we  write $\tau_1\lesssim\tau_2$ if there are $C, D > 0$ such that $\tau_1(x)\le C \tau_2(x) + D$ for every $x \in X$.

\begin{thm}\label{expclin}
Let $G$ be a complex Lie group  with finitely many components. If $G$~is linear, then the embedding $\io\!:\cO_{exp}(G)\to\cO(G)$ is an Arens-Michael envelope.
\end{thm}
\begin{proof}
Since $G$ is linear, so is~$G_0$. By Corollary~\ref{expclin-1}, it
is sufficient to show that the restriction map
$\psi_0\!:\cO_{exp}(G)\to \cO_{exp}(G_0)$ has dense range.

In fact, $\psi_0$ is surjective. Indeed, denote by $\mu(f)$ the extension of~$f\in \cO(G_0)$ such that $\mu(f)(g)=0$ when $g\notin G_0$. Since $G/G_0$ is finite, $G_0$ is obviously cocompact in~$G$.
Any cocompact  closed subgroup of a compactly generated locally compact group  is
undistorted \cite[Lemma~1]{Va99}, that is, $\ell_0\lesssim\ell$ for any word length functions $\ell_0$ and $\ell$ on $G_0$ and $G$, respectively. Therefore we have from
\eqref{exptwlf} that  $\mu(f)\in\cO_{exp}(G)$  for any $f\in \cO_{exp}(G_0)$. Note that $\psi_0\mu$ is the identity map on $\cO_{exp}(G_0)$, so $\psi_0$ is surjective.
\end{proof}

For a complex Lie group  $G$, we consider the quotient homomorphism $\pi\!:G\to G/\LinC(G)$ and also the
natural embeddings
$$
\io\!:\cO_{exp}(G)\to \cO(G)\quad\text{and}\quad
j\!:\cO_{exp}(G/\LinC(G))\to \cO(G/\LinC(G)).
$$

Recall that a \emph{Stein group} is a complex Lie group such that its underlying complex manifold is a Stein manifold.
The following result extends \cite[Theorem~5.12]{ArAnF} from the connected case to the case of finitely many components.

\begin{thm}\label{OGhrcon}
If $G$ is a complex Lie group with finitely many
components, then

\emph{(A)}~$j\wt\pi^{-1}\!:\cO_{exp}(G)\to \cO(G/\LinC(G))$ is an
Arens-Michael envelope;

\emph{(B)}~$G$ is linear if and only if it is a Stein group and
$\io\!:\cO_{exp}(G)\to \cO(G)$ is an Arens-Michael envelope.
\end{thm}
\begin{proof}
The structure of the argument is the same as in~\cite[Theorem~5.12]{ArAnF}. First, we prove the necessity in Part~(B), then Part~(A) and, finally, the sufficiency in Part~(B).

(1)~If $G$ is linear, then so is~$G_0$. Since any connected linear group is biholomorphically equivalent to an affine algebraic complex variety, we have that $G_0$ is a Stein group and so is~$G$. On the other hand, Theorem~\ref{expclin} implies that $\io$ is an Arens-Michael envelope. Thus the condition in Part~(B) is necessary.

(2)~Write $L\!:=\LinC(G)$. It is
obvious that $G/L$ is linear and have finitely many
components. By Theorem~\ref{Linfcc}, we have $L=\wtLinC(G)$. Since any Lie group with finitely many   components is compactly generated, Theorem~\ref{redliexpf}(B) implies that $\wt\pi$ is a topological isomorphism.  The
above argument shows that $j$ is an Arens-Michael envelope, so
is $j\wt\pi^{-1}$. Part~(A) is proved.

(3)~Suppose
that $G$ is a Stein group and $\io$ is an Arens-Michael envelope.
Since, by Part~(A), so is $j\wt\pi^{-1}$, the universal property of  the
Arens-Michael enveloping functor  implies that $\cO(G/L)\to
\cO(G) $ is a topological isomorphism of Stein algebras. By
Forster's Duality Theorem \cite{Fo67}, the quotient map $G\to
G/L$ is a biholomorphic equivalence, therefore $L$~is trivial. This completes the proof of the sufficiency in Part~(B).
\end{proof}

Note that we have obtained not only a description of the Arens-Michael envelope of $\cO_{exp}(G)$ in our partial case but also a criterion of linearity.

\section{Holomorphic reflexivity}
\label{sec:holrefl}

In this section the main results, two theorems on holomorphic reflexivity, are proved.
First, we consider a compactly generated Stein group~$G$ with the corresponding (commutative nuclear Fr\'echet Arens-Michael) Hopf $\ptn$-algebra $\cO(G)$ and establish necessary conditions for holomorphic reflexivity. We denote, as usual, the component of the identity by~$G_0$.

\begin{pr}\label{tLinness}
Let $G$ be a compactly generated Stein group. If $\cO(G)$ is holomorphically reflexive, then $\wtLinC(G)$ and $\LinC(G_0)$ are trivial.
\end{pr}

We need the following simple lemma.
\begin{lm}\label{nesswtL1}
Let $G$~be a Stein group such that $\cO_{exp}(G)$ is dense in~$\cO(G)$. If~$H$ is a closed subgroup such that each function in~$\cO_{exp}(G)$ is constant on~$H$, then $H=\{1\}$.
\end{lm}
\begin{proof}
Since $\cO_{exp}(G)$ is dense in~$\cO(G)$, each function in  $\cO(G)$ is constant on~$H$.  Being a Stein manifold, $G$ is holomorphically separable; hence $H=\{1\}$.
\end{proof}

\begin{proof}[Proof of Proposition~\ref{tLinness}]
Since $\cO(G)$ is holomorphically reflexive, it follows from Proposition~\ref{AMeqHR} that  the embedding $\cO_{exp}(G)\to\cO(G)$ is an Arens-Michael envelope; in particular, it has dense range. On the other hand, by Theorem~\ref{redliexpf}(C), each function in  $\cO_{exp}(G)$ is constant on~$\wtLinC(G)$. So Lemma~\ref{nesswtL1} implies that $\wtLinC(G)=\{1\}$. Finally, $\LinC(G_0)=\{1\}$ by~Corollary~\ref{LwtL}.
\end{proof}

\begin{thm}\label{OGhrcon1}
Let  $G$ be a Stein group with finitely many
components. The following conditions are equivalent.
\begin{mycompactenum}
\item~$\cO(G)$ is holomorphically reflexive;
\item~$\LinC(G_0)$ is trivial;
\item~$\wtLinC(G)$ is trivial.
\end{mycompactenum}
\end{thm}

\begin{proof}
Since $G$~is compactly generated,  the implications  $\mathrm{(i)}\Longrightarrow\mathrm{(ii)}$ and $\mathrm{(i)}\Longrightarrow\mathrm{(iii)}$ follow from Proposition~\ref{tLinness}.
Theorem~\ref{Linfcc} implies that $\LinC(G_0)=\LinC(G)=\wtLinC(G)$, therefore $\mathrm{(ii)}\Longleftrightarrow\mathrm{(iii)}$. Finally, the implication $\mathrm{(iii)}\Longrightarrow\mathrm{(i)}$ follows from Theorem~\ref{OGhrcon}(B) and Proposition~\ref{AMeqHR}.
\end{proof}
\begin{rems}
(A)~Since we  assume that the group has finitely many
components, Conditions~(ii) and~(iii) in Theorem~\ref{OGhrcon1} can be formulated more traditionally:~$G_0$ is linear and~$G$ is linear, respectively. (The latter because $\wtLinC(G)=\LinC(G)$ by Theorem~\ref{Linfcc}.) But I guess that the formulation in the theorem is adequate and the equivalence $\mathrm{(i)}\Longleftrightarrow\mathrm{(iii)}$ holds for all compactly generated Stein groups (as formulated in Conjecture~\ref{mainconj}).

(B)~Since any linear group is a Stein group, the implication  $\mathrm{(ii)}\Longrightarrow\mathrm{(i)}$  in  Theorem~\ref{OGhrcon1} is  true under the assumption that $G$ has finitely many
components. But the converse does not hold in general. For example, for a compact torus~$T$, we have $\cO(T)\cong\CC$, so $\cO(T)$ is holomorphically reflexive. But it is obvious that $T$~is  not linear.

(C)~Example~\ref{countex} shows that $\mathrm{(ii)}$ does not imply
$\mathrm{(i)}$ when $G/G_0$ is infinite.
\end{rems}

In the case of finitely many
components we prove the holomorphic reflexivity of the (cocommutative nuclear Fr\'echet Arens-Michael) Hopf $\ptn$-algebra $\cA_{exp}(G)$ without the assumption that $G$~is a linear group. First, we establish a more general result.

\begin{thm}\label{AeGhr-1}
Let $G$ be a compactly generated complex Lie group. If $\cO(G/\wtLinC(G))$ is holomorphically reflexive, then so is~$\cA_{exp}(G)$.
\end{thm}
\begin{proof}
We have to show that the map
$\ka\!:\cA_{exp}(G)^{\bullet\bullet}\to \cA_{exp}(G)$ defined in~\eqref{defioH} is an isomorphism.

Put $\wt L\!:=\wtLinC(G)$.
Since $\cA_{exp}(G)'\cong \cO_{exp}(G)$, it follows from the definition of the holomorphic dual that $\cA_{exp}(G)^\bullet\cong \wh\cO_{exp}(G)$. Moreover,
$\cO_{exp}(G/\wt L)\cong \cO_{exp}(G) $  by Theorem~\ref{redliexpf}. The assumption that $\cO(G/\wt L)$ is holomorphically reflexive and Proposition~\ref{AMeqHR}
imply that $\wh\cO_{exp}(G)\cong
\cO(G/\wt L)$. Hence $\cA_{exp}(G)^\bullet\cong\cO(G/\wt L)$.

On the other hand, $\cO(G/\wt L)'\cong\cA(G/\wt L)$ and,
by Akbarov's result (Theorem~\ref{AMEAG}), $\cA(G/\wt L)\to \cA_{exp}(G/\wt L)$ is an Arens-Michael envelope.
It follows from
Theorem~\ref{redliexpf} that the naturally defined map $ \cA_{exp}(G)\to \cA_{exp}(G/\wt L)$
is an isomorphism. So
$\cO(G/\wt L)^\bullet\cong \cA_{exp}(G)$ and therefore $\ka$ is an isomorphism.
\end{proof}

\begin{co}\label{AeGhr}
If $G$ is a  complex Lie group with finitely many
components, then $\cA_{exp}(G)$ is holomorphically reflexive.
\end{co}
\begin{proof}
If follows from Theorems~\ref{Linfcc} and~\ref{OGhrcon} that $j\wt\pi^{-1}\!:\cO_{exp}(G)\to \cO(G/\wtLinC(G))$ is an Arens-Michael envelope.  So, applying Theorem~\ref{redliexpf} we get  that $\cO(G/\wtLinC(G))$ is holomorphically reflexive. Since $G$ is compactly generated, Theorem~\ref{AeGhr-1} implies that $\cA_{exp}(G)$ is also holomorphically reflexive.
\end{proof}

It is possible that the  assumption of finiteness of the set of components can be removed (see Conjecture~\ref{mainconj2}).

\section{Examples}
\label{sec:exmp}

In this section we consider examples of  compactly generated complex Lie groups with infinitely many components. First, we make some additional work to show that the group~$G$ considered in Example~\ref{countex}  gives a counterexample to Akbarov's conjecture: the component of the identity of~$G$ is linear but $\cO(G)$ is not holomorphically reflexive. Secondly,  we exhibit  a family of examples of~$G$  with infinitely many components such that $\cO(G)$ is holomorphically reflexive.

Recall that a closed subgroup~$H$ of a compactly generated locally compact group~$G$ is said to be \emph{distorted} if $\ell_H\not\lesssim\ell_G$ for any word length functions $\ell_H$ and $\ell_G$ on $H$ and $G$, respectively.

\begin{pr}\label{Ctiescons}
Let $G$ be a compactly generated complex Lie group.
If $H$ is a distorted closed subgroup of $G$ and $H$ is isomorphic to~$\CC^\times$, then each function from $\cO_{exp}(G)$ is constant on~$H$.
\end{pr}
\begin{proof}
Let $f\in\cO_{exp}(G)$.
Consider the holomorphic homomorphism $\mu\!:\CC^\times\to G$ that implements the isomorphism $\CC^\times\to H$.  Since  $f\mu$  is holomorphic on $\CC^\times$, we can write the Laurent series: $f\mu(z)=\sum_{m\in\Z}c_m\,z^m$.
To show that $f$ is constant on~$H$ it suffices to prove that $c_m=0$ for each $m\ne 0$.

Let $U$ be a compact generating subset of~$G$ and $\om$ the corresponding word submultiplicative weight, i.e., the exponent of the word length function. We can assume that
$\mu(\mathbb{T})$ (where $\mathbb{T}=\{z\!:\,|z|=1\}$) is contained in~$U$. Then
$\om\mu(\la)=1$ for any $\la\in \mathbb{T}\setminus \{1\}$ and, by the submultiplicativity,
\begin{equation}\label{ommumodz}
 \om\mu(z)= \om\mu(|z|)\qquad \text{for each $z\in \CC$.}
\end{equation}

Since $f$ is of exponential type and $\om$ is a word weight, there are $C>0$ and $k\in\N$ such that
$$
|f(g)|\le C\,\om^k(g)\qquad \text{for each $g\in G$.}
$$
Write $M_R\!:=\max\{|f\mu(z)|\!:\,|z|=R\}$. Then, by~\eqref{ommumodz}, we have for every $R>0$ that
\begin{equation}\label{MromR}
M_R\le \max\{C\,\om^k\mu(z)\!:\,|z|=R\}= C \om^k\mu(R)\,.
\end{equation}

Note that $z\mapsto \max\{|z|,\,|z|^{-1}\}$ is a submultiplicative weight on~$\CC^\times$ such that the corresponding length function is equivalent to a word length function. Since $H$ is distorted in~$G$,
there is a sequence $(z_n)$ in~$\CC^\times$ such that
\begin{equation}\label{omnmuzn}
 \max\{|z_n|,\,|z_n|^{-1}\}>\om^n\mu(z_n)\qquad \text{for each $n\in\N$.}
\end{equation}
Passing to a subsequence if necessary, we can assume that $(z_n)$ has a limit, finite or infinite, and~\eqref{omnmuzn} is still satisfied.

We claim first that $\lim_{n\to\infty} z_n$ does not belong to~$\CC^\times$. Indeed, assume the opposite. Then there is $K>0$ such that $\max\{|z_{n}|,\,|z_{n}|^{-1}\}\le K$ and therefore, by~\eqref{omnmuzn}, $\om\mu(z_{n})< K^{1/n}$ for every~$n$. Since  $\om$ takes values in $\{e^p\!:\,p\in\Z_+\}$, we have that $\om\mu(z_{n})=1$ eventually. By the definition of a word weight, $\om(g)=1$ if only if $g=1$. Since $\mu$ is injective, $z_{n}=1$ eventually. Applying again \eqref{omnmuzn} we have that $\om(1)<1$ and get a contradiction.

Thus only two cases may occur: $\lim z_n=\infty$ and $\lim z_n=0$.  Suppose that $\lim z_n=\infty$ and write $R_n\!:=|z_n|$. Applying consequently~\eqref{MromR} and~\eqref{omnmuzn} we get
$$
M_{R_n}\le  C \om^k\mu(R_n)< C \max\{R_n^{k/n},\,R_n^{-k/n}\} \,.
$$

Now consider the coefficient $c_m$ of the Laurent series
for $m\ne 0$. If $n\ge k(|m|-1/2)^{-1}$, then $|m|-k/n\ge 1/2$ and hence $R^{k/n-|m|}\le R^{-1/2}$ when $R\ge 1$.
Since $R_n\ge 1$ eventually,
$$
\frac{M_{R_n}}{R_n^{|m|}} \le C\,R_n^{k/n-|m|}\le C\,R_n^{-1/2}
$$
also eventually. Finally, we have
$\lim_{n\to \infty} M_{R_n}/R_n^{|m|}=0$ since $\lim R_n=\infty$.
Cauchy's inequality for Laurent series asserts that
$|c_m|\le M_R/R^m$ for any $R>0$ and $m\in\Z$, therefore, putting $R=R_n$ for $m>0$ and $R=R_n^{-1}$ for $m<0$, we have $c_m=0$.

The proof in the case when $\lim z_n=0$ is similar.
\end{proof}

\begin{co}\label{Ctimescoin}
If~$G$ is the group in Example~\ref{countex}, then its component of the identity is linear but $\cO(G)$ is not holomorphically reflexive. Nevertheless, $\cA_{exp}(G)$ is holomorphically reflexive.
\end{co}
\begin{proof}
The component~$G_0$  of the identity is isomorphic to~$\CC^\times$ and distorted in $G$ (in fact, it is of strictly quadratic distortion). By Proposition~\ref{Ctiescons}, each holomorphic function of exponential type on $G$ is constant on~$G_0$.

Suppose to the contrary that $\cO(G)$ is holomorphically reflexive. Since~$G$ is compactly generated, then, by Proposition~\ref{AMeqHR}, the embedding $\cO_{exp}(G)\to \cO(G)$ is an Arens-Michael
envelope. In particular, $\cO_{exp}(G)$ is dense in~$\cO(G)$.  Lemma~\ref{nesswtL1} implies that $G_0$ is trivial and we get a contradiction.

Since each function from $\cO_{exp}(G)$ is constant on~$G_0$,
it follows from Corollary~\ref{coredliexpf} that $G_0\subset \wtLinC(G)$. So, by Lemma~\ref{lemmcenter}, $G_0= \wtLinC(G)$. Part~(B) of Theorem~\ref{redliexpf} implies that $\cO_{exp}(G)\cong \cO_{exp}(\Ga)$, where $\Ga=G/G_0$.
Since $\Ga$ is discrete and finitely generated, $\cO_{exp}(\Ga)\cong \cO(\Ga)$ by Proposition~\ref{expdis} and so $\cO(\Ga)$ is  holomorphically reflexive. Finally,
Theorem~\ref{AeGhr-1} implies that $\cA_{exp}(G)$ is holomorphically reflexive.
\end{proof}

\begin{pr}\label{Pnilsdpr}
Suppose that a complex Lie group~$G$ is a semidirect product
$G=N\rtimes H$, where~$H $ is a compactly generated closed subgroup and~$N$ is a
closed normal subgroup that is simply connected and nilpotent.
Let~$\ell$ be a word length function on~$G$, let~$\fn$  denote the
Lie algebra of~$N$, and let~$\exp$ be the exponential map on~$\fn$. Given
a norm $\|\cdot\|$ on~$\fn$, there are $C,D>0$ such that
\begin{equation*}
\log (1+\|\eta\|)\lesssim \ell(\exp \eta)\qquad\text{for $\eta\in \fn$.}
\end{equation*}
\end{pr}
\begin{proof}
Let $K$ be a compact generating subset of~$H$. Then $$S\!:=\{\exp \eta\!:\,\|\eta\|\le 1\}\times K$$ is  compact and generates~$G$.
We can assume that $\ell$ is the word length function associated with~$S$.

Denote by $\al_h$ the action of $h\in H$ on $ N$ and by $\mathbf{L}(\al_h)$ the corresponding automorphism of~$\fn$. Since~$K$ is compact, there is $B\ge0$ such that $\|\mathbf{L}(\al_h)\|\le B$ for all $h\in K$. We assume that $B>1$ for technical reasons.

Denote by $\ell_N$ a word length function on~$N$. It is a standard result of the geometric theory of nilpotent Lie groups (see, e.g., \cite[Theorem~3.1]{ArAMN}) that
there are $A,A',k,k'\ge 1$ such that
\begin{equation}\label{AAprkk}
\|\eta\|\le A'(1+\ell_N(\exp\eta))^{k'}\quad\text{and}\quad \ell_N(\exp\eta)\le A(1+\|\eta\|)^k
\end{equation}
for every $\eta\in\fn$. The first inequality and $A'\ge 1$ imply that
$$
\log (1+\|\exp^{-1}(n)\|)\lesssim \log\ell_N(n)\,
$$
for $n\in N\setminus\{1\}$. So it suffices to show that $\log\ell_N(n)\lesssim \ell(n)$.

Since $N$ is simply connected and nilpotent, the exponential map is bijective.  Since the exponential map is natural, we have
\begin{equation}\label{expnat}
\exp^{-1}\circ\,\al_h\circ\exp(\eta)=\mathbf{L}(\al_h)(\eta)\qquad
\text{for each $\eta\in\fn$ and $h\in H$.}
\end{equation}

Suppose that $n\in N$ with $\ell(n)=m$, where $m>0$, and fix a representation $n=g_1\cdots g_m$, where $g_1,\ldots, g_m\in S$, i.e., $g_j=(\exp \eta_j,\,h_j)$ for some $h_j\in K$ and $\eta_j\in\fn$ with $\|\eta_j\|\le 1$.
Then
$$
n=(\exp \eta_1)\,\al_{h_1}(\exp \eta_2)\cdots\al_{h_1\cdots h_{m-1}}(\exp \eta_m)\,.
$$

Put $h_0\!:=1$. Then $\|\mathbf{L}(\al_{h_0\cdots h_{j}}(\eta_j))\|\le B^j$ for any $j\in\{0,\ldots,m-1\}$.
Applying \eqref{expnat} and the second inequality in~\eqref{AAprkk} we get
\begin{multline*}
 \ell_N(n)\le \sum_{j=0}^{m-1}\ell_N(\al_{h_0\cdots h_{j}}(\exp \eta_j))\le\\
 \sum_{j=0}^{m-1} A(1+\|\mathbf{L}(\al_{h_0\cdots h_{j}}(\eta_j))\|)^k\le
 A\,\sum_{j=0}^{m-1} (1+B^j)^k\le\\
 A\,2^k\,\sum_{j=0}^{m-1} B^{jk}\le
 \frac{A\,2^k\, B^{mk}}{B^k-1}\,.
\end{multline*}
Since $m=\ell(n)$, we have $\log\ell_N(n)\lesssim \ell(n)$.
\end{proof}

\begin{exm}\label{infcomphr}
Let $N$ be a complex Lie group that is simply connected and
nilpotent and let~$\Ga$ be a finitely generated subgroup of the
automorphism group of~$N$. When~$\Ga$ is endowed with the discrete
topology, $G\!:=N\rtimes \Ga$ is a compactly generated complex Lie
group and~$N$ is the component of the identity. When~$\Ga$ is infinite, $G$ has infinitely many components.

We claim that $\cO(G)$ is
holomorphically reflexive.
Indeed, by Corollary~\ref{expclin-1},  it suffices to show that the
homomorphism $\psi_0\!:\cO_{exp}(G)\to \cO_{exp}(N)$ induced by
the embedding $N\to G$ has dense range.
Denote $\fn$ by the Lie algebra of~$N$. Since~$N$ is simply
connected and nilpotent, the exponential map $\exp\!:\fn\to N$ is
a biholomorphic equivalence. Let $\mu(f)$ denote the extension
of~$f\in \cO(N)$ such that $\mu(f)(g)=0$ when $g\notin N$. If  $\|\cdot\|$ is a
norm on~$\fn$, then for any polynomial~$p$ on~$\fn$ there
is $C>0$ and $k\in\Z_+$ such that $|p(\eta)|\le C(1+\|\eta\|)^k$,
and Proposition~\ref{Pnilsdpr} implies that $\mu(p\circ\exp^{-1})\in \cO_{exp}(G)$.

The algebra of holomorphic functions of exponential type on a simply connected  nilpotent complex Lie group is described explicitly in \cite[Theorem~5.7]{ArAnF}.
It follows from this description  that the set of functions of the form
$p\circ\exp^{-1}$, where~$p$ is a polynomial on~$\fn$, is contained and dense in
$\cO_{exp}(N)$.  Then the equality $\psi_0\mu(p\circ\exp^{-1})=p\circ\exp^{-1}$ implies that $\psi_0$ has
dense range.
\end{exm}

\section{Appendix. The Arens-Michael envelope of a nuclear $\ptn$-algebra}

In this section we show that the functor $H\mapsto H^\bullet$ preserves nuclearity when $H$ is an NFAM Hopf algebra.
To do this we need the construction of the analytic tensor algebra of a complete locally convex space considered in \cite{Pir_qfree}. (In fact, we need a partial case because in [ibid.]  the case of a $\ptn$-bimodule is discussed.)

Let $X$ be a complete locally convex space. Then there exists an Arens-Michael algebra $\wh T(X)$ (the analytic tensor algebra of~$X$) and a continuous linear map $j_X\!:X\to \wh T(X)$ satisfying the following universal property:
for any Arens–Michael algebra~$B$ and any continuous linear map $\al\!: X \to B$
there exists a unique continuous homomorphism $\psi\!: \wh T(X) \to B$ making the diagram
\begin{equation*}
  \xymatrix{
X \ar[r]^{j_X}\ar[rd]_\al&\wh T(X)\ar@{-->}[d]^{\psi}\\
 &B\\
 }
\end{equation*}
commutative \cite[Proposition~4.8]{Pir_qfree}.

The proof of following proposition is straightforward (cf. \cite[Proposition~4.2]{Pi15}).

\begin{pr}\label{AMtensa}
Let $A$ be a $\ptn$-algebra and let $\wt A$ be the
completion of the quotient of $\wh T(A)$  by the two-sided closed ideal $J$ generated by  $1_{\wh T}-1_A$
and  all elements of the form $a_1\otimes a_2-a_1a_2$ ($a_1,a_2\in A$). Then the composition of $j_A\!: A\to \wh T(A)$ with the natural map $\wh T(A)\to\wt A $ is an Arens-Michael envelope of~$A$.
\end{pr}

The following result is proved the case of a Fr\'echet space in \cite[Theorem 3.1]{Vo09}.

\begin{pr}\label{nucltensal}
If $X$ is nuclear then $\wh T(X)$ is nuclear.
\end{pr}
\begin{proof}
We need the explicit construction of $\wh T(X)$  from \cite{Pir_qfree}.
Let $\{\|\cdot\|_\nu\!:\, \nu \in \La\}$ be a directed system of prenorms defining the topology on~$X$. Denote by $\|\cdot\|_\nu^{\otimes n}$  the projective tensor product of~$n$ copies of $\|\cdot\|_\nu$.
Then
$\wh T(X)$ consists of all sequences $(x_n;\,n\in\Z_+)$, where $x_n\in X^{\ptn n}$ (with $X^{\ptn 0}\!:=\CC$), such that for every $\nu \in \La$ and $\rho>0$
$$
\|x\|_{\nu,\rho}\!:=\sum_{n=0}^{\infty} \|x_n\|_\nu^{\otimes n}\rho^n<\infty.
$$
%Note that $j$ maps $x$ to the sequence with $x$ at the 1st place and zeros at other places.

Denote by $X_{\nu,\rho}$ the completion of $X$ with respect to $\|\cdot\|_{\nu,\rho}$ and for $\nu'\ge \nu$ and  $\rho'\ge \rho$ denote by $J_{\nu'\rho'}^{\nu\rho}$ the corresponding linking map $X_{\nu',\rho'}\to X_{\nu,\rho}$.
To prove that $\wh T(X)$ is nuclear it suffices to show that for any $\nu$ and $\rho$ there are $\nu'\ge \nu$ and  $\rho'\ge \rho$ such that $J_{\nu'\rho'}^{\nu\rho}$ is a nuclear map.

Denote also by $X_{\nu}$ the completion of $X$ with respect to $\|\cdot\|_{\nu}$ and by $j_{\nu'}^{\nu}$ the linking maps.
Since $X$ is nuclear, there is $\nu'\ge \nu$ such that $j_{\nu'}^{\nu}$ is a nuclear map. Denote the nuclear norm of $j_{\nu'}^{\nu}$ by~$N$.
Then for every $n\in\Z_+$ the map
$(j_{\nu'}^{\nu})^{\otimes n}\!: X_{\nu',\rho'}^{\ptn n}\to X_{\nu,\rho}^{\ptn n}$ is a nuclear map with nuclear norm at most~$N^n$.
Take $\rho'$ such that $\rho'>N \rho$.  A~straightforward calculation shows that $J_{\nu'\rho'}^{\nu\rho}$ is nuclear  with nuclear norm at most~$1$.
\end{proof}

\begin{thm}\label{nuclAMe}
If  a $\ptn$-algebra $A$ is nuclear, then so is its Arens-Michael envelope.
\end{thm}
\begin{proof}
By Proposition~\ref{nucltensal},  $\wh T(A)$ is nuclear. Since the class of nuclear spaces is stable under quotients by closed subspaces and under completions, $\wt A$ is nuclear. An application of  Proposition~\ref{AMtensa} completes the proof.
\end{proof}

\begin{co}\label{conuclAMe}
If $H$ is an NFAM Hopf algebra, then
$H^\bullet$  is nuclear.
\end{co}
\begin{proof}
Since $H$ is a nuclear Fr\'echet space, $H'$ is a nuclear $\ptn$-algebra.
By~\eqref{defhold}, $H^\bullet\!:=(H')\sphat$~. So we can apply Theorem~\ref{nuclAMe}.
\end{proof}

\end{document}